\documentclass[12pt]{amsart}

\usepackage{amssymb}
\usepackage{color}
\usepackage{enumerate}
\usepackage[title,titletoc,toc]{appendix}

\usepackage{graphicx}

\makeatletter
\@namedef{subjclassname@2010}{%
  \textup{2010} Mathematics Subject Classification}
\makeatother

\input xy
\xyoption{all}

\newcommand{\p}{\mathbb{P}}

\newcommand{\F}{\mathbb{F}}

\newcommand{\lra}{\longrightarrow}

\newcommand{\Lra}{\Longrightarrow}

\def \P {\mathbb{P}}

\newcommand{\ff}{\mathcal{F}}
\newcommand{\hh}{\mathcal{H}}
\newcommand{\qq}{\mathcal{Q}}
\newcommand{\G}{\mathcal{G}}
\newcommand{\fq}{\mathbb{F}_q}

\newcommand{\D}{\mathcal D}

\newcommand{\kk}{\mathbb K}
\newcommand{\xx}{\mathcal X}

\newcommand{\cc}{\mathcal C}

\newtheorem{thm}{Theorem}[section]
\newtheorem{cor}[thm]{Corollary}
\newtheorem{lem}[thm]{Lemma}
\newtheorem{prop}[thm]{Proposition}

\theoremstyle{definition}

\newtheorem{rem}[thm]{Remark}
\newtheorem{exa}[thm]{Example}

\numberwithin{equation}{section}

\begin{document}


\baselineskip=17pt



\title[Frobenius nonclassical curves]{Frobenius nonclassicality with respect to linear systems of curves of arbitrary degree}

\author[N. Arakelian]{Nazar Arakelian}
\address{Instituto de Matem\'atica, Estat\'istica e Computa\c c\~ao Cient\'ifica\\ Universidade Estadual de Campinas\\
Rua S\'ergio Buarque de Holanda, 651, CEP 13083-859, Campinas SP, Brazil}
\email{nazar@ime.unicamp.br}

\author[H.Borges]{Herivelto Borges}
\address{Instituto de Ci\^encias Matem\'aticas e de Computa\c c\~ao\\ Universidade de S\~ao Paulo\\
Avenida Trabalhador S\~ao-carlense, 400, CEP 13566-590, S\~ao Carlos SP, Brazil}
\email{hborges@icmc.usp.br}

\date{}

\begin{abstract}
For each integer $s \geq 1$, we present a family of curves that are $\fq$-Frobenius nonclassical   with respect to the linear system of plane curves of degree $s$. In the case $s=2$, we give necessary and sufficient conditions for such curves to be $\fq$-Frobenius nonclassical with respect to the linear system of conics. In the $\fq$-Frobenius nonclassical cases, we determine the exact number of $\fq$-rational points. In the remaining cases, an  upper bound for the number of $\fq$-rational points will follow from  St\"ohr-Voloch theory.

 \end{abstract}



\maketitle

\section {Introduction}

Let $p$ be a prime integer and $\fq$ be a finite field with $q=p^h$ elements. The problem of estimating the number of rational points on curves over $\fq$ has been extensively investigated in view of its broad relevance and application, e.g.,   in finite geometry, number theory, coding theory, etc., see \cite{Hi},\cite{GL},\cite[Chapter 6]{LN} and \cite[Chapters 2 and 8]{St}. Like studying curves with many $\fq$-rational points, it also poses an interesting problem.

 Let $\xx$ be a projective, nonsingular, geometrically irreducible curve of genus $g$ defined over $\fq$, and let $N_q(\xx)$ be its number of $\fq$-rational points. The most remarkable result regarding $N_q(\xx)$ is the Hasse-Weil bound, which states that
\begin{equation}\label{hasse-weil}
|N_q(\xx)-(q+1)| \leq 2g\sqrt{q}.
\end{equation}
In 1986, St\"ohr and Voloch introduced a technique to estimate $N_q(\xx)$, which is dependent on the morphisms $\phi: \xx \rightarrow \mathbb{P}^n$ \cite{SV}.  In many instances, their results improve the Hasse-Weil bound (\cite{SV},\cite{GV2}).

In this paper, we consider  a  family of curves $\xx$ and  focus on aspects   relevant to the application of the St\"ohr-Voloch theory, addresing   the Frobenius (non)classicality of $\xx$ with respect to  linear systems of curves degree $s\geq 1$.

Let $F(x,y,z) \in \fq[x,y,z]$ be a homogeneous polynomial, such that 
$$\xx:F(x,y,z)=0$$ is a nonsingular projective plane curve of degree $d$ and genus $g$. Associated with the linear system of all plane curves of degree $s \in \{1,\ldots,d-1\}$, the curve $\xx$ has a linear series $\D_s$ of dimension $M={s+2 \choose 2}-1$ and degree $sd$ \cite[Section 7.7]{HKT}. Applying St\"ohr-Voloch's theorem \cite[Theorem 2.13]{SV} to $\D_s$ yields
\begin{equation}\label{SV-s}
N_q(\xx) \leq \displaystyle\frac{d(d-3)(\nu_1+\cdots +\nu_{M-1})+sd(q+M)}{M},
\end{equation}
where $(\nu_0,\ldots,\nu_{M-1})$ is the $\fq$-Frobenius order sequence of $\xx$ with respect to $\D_s$.
The curve $\xx$ is called $\fq$-Frobenius classical with respect to $\D_s$ if $\nu_i=i$ for all $i=0,\ldots,M-1$. Note that for such a curve, the bound (\ref{SV-s}) reads
\begin{equation}\label{SV-s2}
N_q(\xx) \leq \displaystyle\frac{d(d-3)(M-1)}{2}+\displaystyle\frac{sd(q+M)}{M}.
\end{equation}
The bound (\ref{SV-s2}) improves the Hasse-Weil bound in several cases (\cite[section 3]{SV}, \cite{GV2}).

If $\nu_i \neq i$ for some $i$, then $\xx$ is called $\fq$-Frobenius nonclassical with respect to $\D_s$. Note that for this case, we have
$$\nu_1+\cdots +\nu_{M-1}>M(M-1)/2.$$
Thus (\ref{SV-s})  indicates that Frobenius nonclassical curves are likely to have many rational points. Therefore, if we can identify  the Frobenius nonclassical curves with respect to $\D_s$, we are left with the remaining curves for which a better upper bound, given by (\ref{SV-s2}), holds.  At the same time, the set of Frobenius nonclassical curves provides  a potential source of curves with many points. Therefore, in  light of (\ref{SV-s}), characterizing Frobenius nonclassical curves may offer a two-fold benefit.

In general, the effectiveness of  (\ref{SV-s2}) will  vary according to the value of $s \in \{1,\ldots,d-1\}$. For instance, in the cases  $s=1$ and $s=2$, the bound (\ref{SV-s2}) reads
\begin{equation}\label{SV-1}
N_q(\xx) \leq \displaystyle\frac{d(d+q-1)}{2}
\end{equation}
and
\begin{equation}\label{SV-2}
N_q(\xx) \leq \frac{2d(5d+q-10)}{5},
\end{equation}
respectively. Note that the bound (\ref{SV-2}) is better than (\ref{SV-1}) when, roughly, $d<q/15$. More generally, if $r \geq 1$,  then the bound (\ref{SV-s2}) for $s=r+1$ is better than the corresponding bound for $s=r$ when, roughly,
\begin{equation}
d<\left(\displaystyle\frac{4}{(r+2)(r+3)(r+4)}\right)q.
\end{equation}
These facts can be interpreted as follows. If we want to find plane curves  of degree $d<q/15$ attaining the bound (\ref{SV-1}),  we must look for plane curves  that are $\fq$-Frobenius nonclassical with respect to $\D_2$. Similarly,  plane curves of degree $d<q/30$  attaining  the bound (\ref{SV-2}) must be $\fq$-Frobenius nonclassical with respect to $\D_3$, and so on. An explicit example of this phenomenon is given in Section $3$. This also highlights the importance of Frobenius nonclassical curves for the construction of curves with many points.

Frobenius (non)classicality in the case $s=1$ has been widely  investigated with many examples cited in the literature (\cite{Bo2},\cite{Ga},\cite{GV2},\cite{HV}). Even for this case, however, a complete  characterization of $\fq$-Frobenius nonclassical curves is lacking. As observed by Hefez and Voloch \cite{HV}, characterizing all such curves seems a quite complex  problem.

In 1988, Garcia and Voloch established necessary and sufficient conditions for a Fermat curve, i.e., a curve given by an equation of the type $ax^d+by^d=z^d$, $a,b \in \fq$, $ab \neq 0$, to be $\fq$-Frobenius nonclassical in the cases $s=1$ and $s=2$ \cite{GV2}. It seems that, excluding the Fermat curves, not many $\fq$-Frobenius nonclassical curves with respect to the linear system of conics are characterized. 

In this paper, we study the $\fq$-Frobenius (non)classicality of a generalization of the Fermat curve. More specifically, we study the smooth projective plane curves $\mathcal{X}$ of degree $d=sn$, defined over $\fq$, and  given by the equation $F(x,y,z)=0$, where 
\begin{equation}\label{x}
F(x,y,z)=\displaystyle\sum_{ i+j+t = s}c_{ij}x^{in}y^{jn}z^{tn},
\end{equation}
with $s \geq 1$ and $n\geq 2$.

The paper proceeds as follows. In Section 2, we set some notation and recall the main results of the St\"ohr-Voloch theory, which constitute the basis for this study.  In Section 3, we provide criteria for the curves arising from (\ref{x}) to be $\fq$-Frobenius nonclassical with respect to the linear series $\D_s$. Then we take advantage of these criteria to  construct new curves of degree $d<q/15$ attaining the St\"ohr-Voloch bound (\ref{SV-1}). In Section 4, we fully characterize the $\fq$-Frobenius nonclassical curves arising from (\ref{x}) in the case $s=2$. In Section 5, we determine the exact value of $N_q(\xx)$, when $\xx$ is an  $\fq$-Frobenius nonclassical curve and, via  St\"ohr-Voloch theory, arrive at a nice upper bound for the number of $\fq$-rational points on the remaining curves.

 The paper's  appendix provides facts about the  irreducibility of some plane quartics. The results listed there  are useful in certain  proofs of Section 4.

\text{}\\

\textbf{Notation}

\text{}\\
Hereafter, we  use the following notation: 

\begin{itemize}
\item $\fq$ is the finite field with $q=p^h$ elements, with $h \geq 1$, for a prime integer $p$.
\item $\kk$ is the algebraic closure of $\fq$.
\item Given an irreducible curve $\xx$ over $\fq$ and an algebraic extension $\mathbb{H}$ of $\fq$, the function field of $\xx$ over $\mathbb{H}$ is denoted by $\mathbb{H}(\xx)$.
\item For a curve $\xx$ and $r>0$, the set of its $\F_{q^r}$-rational points is denoted by $\xx(\F_{q^r})$.
\item $N_{q^r}(\xx)$ is the number of $\F_{q^r}$-rational points of the curve $\xx$.
\item For a nonsingular point $P \in \xx$, the discrete valuation at $P$ is denoted by $v_P$.
\item For two plane curves $\xx$ and $\mathcal{Y}$, the intersection multiplicity of $\xx$ and $\mathcal{Y}$ at the point $P$ is denoted by $I(P,\xx \cap \mathcal{Y})$.
\item Given $g \in \kk(\xx)$, $t$ a separating variable of $\kk(\xx)$ and $r \geq 0$, the $r$-th Hasse derivative of $g$ with respect to $t$ is denoted by  $D_t^{(r)}g$.
\end{itemize}


\section{Preliminaries}

In this section, we recall  results from \cite{SV}. Let $\xx$ be a projective, irreducible, nonsingular curve of genus $g$ defined over $\fq$. Associated to a nondegenerated morphism $\phi=(f_0:\ldots:f_n): \xx \lra \p^{n}(\kk)$, there exists a base-point-free linear series given by
$$
\D_\phi=\left\{div\left(\displaystyle\sum_{i=0}^{n}a_if_i\right)+E \ | \ a_0,\ldots,a_n \in \kk\right\},
$$
with $E:=\sum\limits_{P \in \xx}e_PP$ and $e_P=-min\{v_P(f_0),\ldots,v_P(f_n)\}$. Given a point $P \in \xx$, there exists a sequence of non-negative integers $(j_0(P),\ldots,j_n(P))$, such that $j_0(P)<\cdots <j_n(P)$, called order sequence of $P$ with respect to $\phi$, which is defined by the numbers $j\geq0$ such that $v_P(D)=j$ for some $D \in \D_\phi$. Except for a finite number of points of $\xx$, the order sequence is the same, and  denoted by $(\epsilon_0,\ldots,\epsilon_n)$. This sequence can also be defined by the minimal sequence, with respect to the lexicographic order, for which
$$
\det(D_t^{(\epsilon_i)}f_j)_{0 \leq i,j \leq n} \neq 0,
$$
where $t$ is a separating variable of $\kk(\xx)$. Moreover, for each $P\in \xx$, 
\begin{equation}\label{ej}
 \epsilon_i\leq j_i (P) \text{  for all  } i \in \{0,\ldots,n\}.
\end{equation}

The curve $\xx$ is called classical with respect to $\phi$ (or $\D_\phi$) if the sequence $(\epsilon_0,\ldots,\epsilon_n)$ is $(0,\ldots,n)$. Otherwise, it is  is called nonclassical.

Let $\kk(\xx)$ be the function field of $\xx$ and define the subfield 
$$(\kk(\xx))_{r}=\{u^{p^{r}}  \ |  \ u \in \kk(\xx)\}.$$
 In \cite[Theorem 1]{GV1} the following criterion is proved, which is  useful in determining whether  $\xx$ is classical with respect to the given morphism.

\begin{thm}\label{GV1-1}
Let $\phi=(f_0:\ldots:f_n): \xx \lra \p^{n}(\kk)$ be a morphism. Then $f_0,\ldots,f_n$ are linearly independent over $(\kk(\xx))_{r}$ if  and only if there exist integers $\epsilon_0,\ldots,\epsilon_n$ with $$0=\epsilon_0<\cdots<\epsilon_n<p^r,$$ 
such that $\det(D_t^{(\epsilon_i)}f_j)_{0 \leq i,j \leq n} \neq 0$.
\end{thm}

Proposition 1.7 in \cite{SV} establishes the following.

\begin{prop}\label{SV1.7}
Let $P \in \xx$ be a point with order sequence $(j_0(P),\ldots,j_n(P))$. If the integer 
$$\prod_{i>r}\frac{j_i(P)-j_r(P)}{i-r}$$ is not divisible by $p$, then $\xx$ is classical with respect to $\D_{\phi}$.
\end{prop}

Now suppose that $\phi$ is defined over $\fq$. The sequence of non-negative integers $(\nu_0,\ldots,\nu_{n-1})$, chosen minimally in the lexicographic order, such that
\begin{equation}\label{fncl}
\left|
  \begin{array}{ccc}
  f_0^q & \ldots & f_n^q \\
  D_t^{(\nu_0)}f_0 & \ldots & D_t^{(\nu_0)}f_n \\
   \vdots & \cdots & \vdots \\
  D_t^{(\nu_{n-1})}f_0 & \cdots & D_t^{(\nu_{n-1})}f_n
  \end{array}
  \right| \neq 0,
  \end{equation}
where $t$ is a separating variable of $\fq(\xx)$, is called the $\fq$-Frobenius sequence of $\xx$ with respect to $\phi$. From \cite[Proposition 2.1]{SV}, we have that $\{\nu_0,\ldots,\nu_{n-1}\}=\{\epsilon_0,\ldots,\epsilon_n\}\backslash\{\epsilon_I\}$ for some $I \in \{1,\ldots,n\}$. If $(\nu_0,\ldots,\nu_{n-1})=(0,\ldots,n-1)$, then the curve $\xx$ is called $\fq$-Frobenius classical with respect to $\phi$. Otherwise, it  is called $\fq$-Frobenius nonclassical.

The following  result   \cite[Remark 8.52]{HKT} shows the close relation between classicality and $\fq$-Frobenius classicality.

\begin{prop}\label{fnc impl nc}
Let $\D$ be a linear series of the curve $\xx$, defined over $\fq$, such that $p>M:=dim(\D)$. If $\xx$ is $\fq$-Frobenius nonclassical with respect to $\D$, then $\xx$ is nonclassical with respect to $\D$.
\end{prop}
 
If  $\xx \subseteq \P^n(\kk)$, the $\fq$-Frobenius map $\Phi_q$ is defined on $\xx$ by
\begin{displaymath}
\begin{array}{cccc}
\Phi_q: &\xx & \lra & \xx \\
        &(a_0:\ldots:a_n) & \longmapsto & (a_0^q:\ldots:a_n^q). 
\end{array}
\end{displaymath}
Note that if $\xx$ is a plane curve, then by (\ref{fncl}) and \cite[Corollary 1.3]{SV}, we have that $\xx$ is $\fq$-Frobenius nonclassical with respect to the linear system of lines if  and only if $\Phi_q(P)$ lies on the tangent line of $\xx$ at $P$ for all $P \in \xx$.

Now let $F(x,y,z) \in \fq[x,y,z]$ be a homogeneous, irreducible polynomial of degree $d$ such that 
$$\xx:F(x,y,z)=0$$
 is a nonsingular projective plane curve. The function field $\kk(\xx)$ is given by $\kk(x,y)$, where $x$ and $y$ are such that $F(x,y,1)=0$. For each $s \in \{1,\ldots,d-1\}$,  consider the Veronese morphism 
$$
\phi_s=(1:x:y:x^2:\ldots:x^{i}y^{j}:\ldots:y^s): \xx \lra \P^M(\kk),
$$
where $i+j\leq s$. It is well known that the linear series $\D_s$ associated with $\phi_s$ is base-point-free of degree $sd$ and dimension $M={s+2 \choose 2}-1=(s^2+3s)/2$. The linear series $\D_s$ is also obtained by the cut out on $\xx$ by the linear system of plane curves of degree $s$.

For  any $P \in \xx$,  a   $(\D_s,P)$-order $j:=j(P)$ can be seen as the intersection multiplicity  at $P$ of $\xx$ with some plane curve of degree $s$. That is, the integers $j_0(P)<\cdots <j_M(P)$ represent   the possible intersection multiplicities of a plane curve of degree $s$ with $\xx$ at $P$.  Moreover, by \cite[Theorem 1.1]{SV}, there is a unique curve $\hh_P^s$ of degree $s$, called $s$-osculating curve of $\xx$ at $P$, such that
$$I(P,\xx \cap \hh_P^s)=j_M(P).$$


\section{$\fq$-Frobenius nonclassical curves}\label{The general case}

Let us recall that $\xx: F(x,y,z)=0$ is a smooth, projective plane curve of degree $sn$, defined over $\fq$, where $F$ is given by 
\begin{equation}\label{x-again}
F(x,y,z)=\displaystyle\sum_{ i+j+t = s}c_{ij}x^{in}y^{jn}z^{tn},
\end{equation}
with $s \geq 1$ and $n \geq 2$. This section establishes sufficient conditions for the curve 
$\xx$ to be $\F_q$-Frobenius nonclassical with respect to $\mathcal{D}_s$. Note that  the case $s=1$ addresses the $\F_q$-Frobenius nonclassicality, with respect to lines,  of Fermat curves
\begin{equation}\label{fermatGV}
\xx: ax^n+by^n+cz^n=0.
\end{equation}
However, for $p\neq 2$, it is a  well known result by Garcia and Voloch \cite[Theorem 2]{GV2})  that the curve (\ref{fermatGV})  is  $\F_q$-Frobenius nonclassical, with respect to lines, if and only if
$n=\frac{p^h-1}{p^v-1}$, and  the curve is defined over $\F_{p^v}$, where $q=p^h$, $v>h$ and  $v|h$. For an alternative proof  including the case $p=2$, see \cite{Bo3}.
 
 Henceforth, we consider a smooth curve $\xx$ associated to $(\ref{x-again})$ with the following assumptions:
\begin{enumerate}
\item[(3.i)] $s\geq 2$
\item[(3.ii)] $p |n-1$
\item[(3.iii)] $p>5$ for $s=2$, and $p>s^2$  for $s\geq 3$ (in particular, $p>M:=\dim \mathcal{D}_s$).
\end{enumerate}

The following  result will be a key ingredient in our approach. It  is proved in \cite[Lemma 1.3.8]{Na} and \cite[Lemma A.2]{Ha},  for curves in characteristics $p=0$ and  $p\geq0 $, respectively. 
\begin{lem}\label{lema4}
Let $\ff$, $\G$ and $\hh$ be three plane curves. If  $\ff$ is nonsingular, then 
 $$I(P,\hh \cap \G) \geq \min \{I(P,\ff \cap \G),I(P,\ff \cap \hh)\}$$
for all $P \in\ff$.
\end{lem}

\begin{lem}\label{lema5}
For all points $P=(a:b:c)\in \xx$ such that $abc \neq 0$, the $s$-osculating curve $\hh_P^{s}$ to $\xx$ at $P$ is an irreducible curve given by the equation $H_P(x,y,z)=0$, where
\begin{eqnarray}\label{curvaHp}
H_P(x,y,z)=\displaystyle\sum_{i+j+t = s}c_{ij}(a^{im}b^{jm}c^{tm})^{p^v}x^iy^jz^t,
\end{eqnarray}
$n=mp^v+1$, and gcd$(p,m)=1$. Furthermore, the curve $\xx$ is nonclassical with respect to $\D_s$ but classical with respect to $\D_{i}$, $1 \leq i \leq s-1$.
\end{lem}
\begin{proof}
Set $f(x,y):=F(x,y,1)$, and note that $f(x,y)=0$ can be written as
\begin{eqnarray}\label{ld}
\displaystyle\sum_{0 \leq i+j \leq s}c_{ij}(x^{im}y^{jm})^{p^v}x^iy^j=0 \in \kk(\xx).
\end{eqnarray}
Therefore, if  $(\epsilon_0,\epsilon_1,\ldots,\epsilon_{M})$ is the $\D_s$-order sequence of $\xx$, then Theorem \ref{GV1-1} implies   $\epsilon_M \geq p^v>M$. Thus  $\xx$ is nonclassical with respect to $\D_s$.  Let $P=(a:b:c)$ be a point of $\xx$, with $abc \neq 0$, and  consider the curve 
$$\cc:H_P(x,y,z)=0$$ of degree $s$ (cf. (\ref{curvaHp})). We claim that $\cc$ is irreducible. To see this,
consider the polynomial $G(x,y,z):=\displaystyle\sum_{i+j+t = s}c_{ij}x^iy^jz^t,$ and note that $$G(a^{mp^v}x,b^{mp^v}y,c^{mp^v}z)=H_P(x,y,z).$$
 Therefore, we  need  only  prove that
the polynomial $G(x,y,z)$ is irreducible. But this follows immediately from the fact $\xx$ is irreducible and  $F(x,y,z)=G(x^n,y^n,z^n)$. We may assume that $P=(a:b:1)$, and then for $h(x,y):=H_P(x,y,1)$,  we have that $h(x,y)=h(x,y)-f(x,y) \in \kk(\xx)$ can be written as
\begin{equation}
h(x,y)=\displaystyle\sum_{0 \leq i+j \leq s}c_{ij}(a^{im}b^{jm}-x^{im}y^{jm})^{p^v}x^iy^j.
\end{equation}
Therefore, $v_P(h(x,y)) \geq p^v$, and then $I(P,\xx \cap \cc) \geq p^v$. Let $\hh_P^{s}$ be the $s$-osculating curve to $\xx$ at $P$. Since $\epsilon_M \geq p^v$, it follows from (\ref{ej}) that 
$$I(P,\xx \cap \hh_P^{s})=j_M(P) \geq p^v .$$  
Thus from Lemma \ref{lema4}, we have $I(P,\cc \cap \hh_P^{s}) \geq p^v$. As we are assuming that $p>s^2$, we have that 
$$I(P,\cc \cap \hh_P^{s}) > s^2=deg(\cc)\cdot deg(\hh_P^{s}).$$
Therefore by B\'ezout's Theorem, the curves $\cc$ and $\hh_P^s$ have a common component. However, since $\cc$ is irreducible and $deg(\cc)=deg(\hh_P^{s})$, it follows that $\cc=\hh_P^{s}$. In particular,  the $s$-osculating curve $\hh_P^{s}$ is irreducible.

For the  lemma's  last statement,  it suffices to prove classicality with respect to $\D_{s-1}$. Suppose that $\xx$ is nonclassical with respect to $\D_{s-1}$. Then by \cite[Corollary 1.9]{SV}, the intersection multiplicity of $\xx$ with the $(s-1)$-osculating curve $\hh_P^{s-1}$ to $\xx$ at any point $P \in \xx$ is $I(P,\xx \cap \hh_P^{s-1}) \geq p$. By Lemma \ref{lema4},  
$$I(P,\hh_P^{s-1} \cap \hh_P^{s}) \geq p >s^2>s(s-1)=deg(\hh_P^{s})\cdot deg(\hh_P^{s-1}),$$ and thus  B\'ezout's Theorem implies that  $\hh_P^{s}$ and $\hh_P^{s-1}$ have a common component. Since this contradicts  the irreducibility of $\hh_P^{s}$, the result follows.
\end{proof}
Next we give the main result of the section.
\begin{thm}\label{teo6}
Let $\hh_P^{s}$ be the $s$-osculating curve to $\xx$ at $P$. Then $\Phi_q(P) \in \hh_P^{s}$ for infinitely many points $P \in \xx$ if  and only if $n=(p^h-1)/(p^v-1)$, and $\xx$ is defined over $\F_{p^v}$, where $q=p^h$, $h>v$ and $v|h$.
\end{thm}
\begin{proof}
Since $p|n-1$, we have that $n=mp^v+1$ for some positive integers $v,m$, where gcd$(p,m)=1$. Suppose that $\Phi_q(P) \in \hh_P^{s}$ for infinitely many points $P \in \xx$. By Lemma \ref{lema5}, this means that the function
\begin{eqnarray}\label{nula}
g(x,y):=\displaystyle\sum_{0 \leq i+j \leq s}c_{ij}(x^{im}y^{jm})^{p^v}x^{iq}y^{jq} \in \kk(\xx)
\end{eqnarray}
is zero, i.e.,  the polynomial $f(x,y):=F(x,y,1)$ divides $g(x,y)$. Since $mp^v+q=n+q-1$, the polynomial $g(x,y)$ can be written as
\begin{equation}\label{polyg}
g(x,y)=\displaystyle\sum_{0 \leq i+j \leq s}c_{ij}x^{i(n+q-1)}y^{j(n+q-1)}.
\end{equation}

Note that $g(x,y)$ is a nonzero polynomial of degree $s(n+q-1)$. Also, it is easy to see that $p^v<q=p^h$, i.e., $v<h$.  Indeed, if $p^v\geq q$, then (\ref{nula}) gives $g(x,y)=l(x,y)^{q}$ where $l(x,y)$ is a polynomial of degree
 $s(n+q-1)/q$. This implies that $f(x,y)$ divides $l(x,y)$, and then
  $$sn=\deg f(x,y)\leq \deg l(x,y)=s(n+q-1)/q,$$
which is impossible for $n>1$.
 
  Therefore, $n+q-1$ is divisible by $p^v$, and then (\ref{polyg}) gives $g(x,y)=r(x,y)^{p^v}$, where
$$
r(x,y)=\displaystyle\sum_{0 \leq i+j \leq s}c_{ij}^{1/p^v}x^{i(m+p^{h-v})}y^{j(m+p^{h-v})}.
$$
 Furthermore, $f(x,y)|r(x,y)$. Now we claim that $r(x,y)$ is irreducible. To see this, let $\mathcal{R}$ be 
  the projective closure  of the curve  $r(x,y)=0$. One can easily  check that if $P=(a:b:c) \in \mathcal{R}$ is a singular point, and $\alpha,\beta,\gamma \in \kk$ are roots of
  $x^n=a^{(m+p^{h-v})p^v}$, $x^n=b^{(m+p^{h-v})p^v}$ and $x^n=c^{(m+p^{h-v})p^v}$, respectively, then $(\alpha:\beta:\gamma)$ is a singular point of $\xx$. However, since $\xx$ is smooth, the curve $\mathcal{R}$  must be smooth, and so $r(x,y)$ is irreducible. This implies   $f(x,y)=\alpha r(x,y)$ for some $\alpha \in \kk^{*}$. Now $\deg f(x,y)=\deg r(x,y)$ gives $n(p^v-1)=p^h-1$, as desired. In addition, $c_{ij}=\alpha {c_{ij}}^{1/p^v}$ for all $i,j$ implies that $c_{ij}/c_{kl} \in \F_{p^v}$ whenever $c_{kl} \neq 0$. That is, the curve  $\xx$ is defined over $\F_{p^v}.$ 

Conversely, suppose that  $n=(p^h-1)/(p^v-1)$, with $h>v$ and  $v|h$, and that $\xx$ is defined over  $\F_{p^v}$. We may assume that all coefficients $c_{ij}$ lie in  $\F_{p^v}$. From Lemma \ref{lema5}, it suffices to prove that  $f(x,y)|g(x,y)$, where $g(x,y)$ is given by (\ref{nula}). Note that  $n+q-1=np^v$, and then (\ref{polyg}) implies $g(x,y)=f(x,y)^{p^v}$, which completes the proof.
\end{proof}

\begin{cor}\label{cor7}
Suppose that $n=(p^h-1)/(p^v-1)$ and that $\xx$ is defined over $\F_{p^v}$, where $h>v$ and $v|h$. Then $\xx$ is $\fq$-Frobenius nonclassical with respect to $\D_s$.
\end{cor}
\begin{proof}
By Theorem \ref{teo6},  $\Phi_q(P) \in \hh_P^{s}$ for infinitely many points $P \in \xx$. Hence, if $\tau$ is a separating variable of $\fq(\xx)$, by \cite[Corollary 1.3]{SV} 
\begin{displaymath}
 \left|
  \begin{array}{ccccc}
  1 & f_1^q &f_2^q&\cdots & f_M^q \\
  1 & f_1 & f_2& \cdots & f_M \\
  0 & D_\tau^{(\epsilon_1)}(f_1) & D_\tau^{(\epsilon_1)}(f_2)& \cdots  &D_\tau^{(\epsilon_1)}(f_M)\\
  \vdots & \vdots & \vdots & \cdots & \vdots \\
  0 & D_\tau^{(\epsilon_{M-1})}(f_1) & D_\tau^{(\epsilon_{M-1})}(f_2) & \cdots  & D_\tau^{(\epsilon_{M-1})}(f_M)\\
  \end{array}
  \right| = 0,
  \end{displaymath}
where $1,f_1,\ldots,f_M$ are the coordinate functions of the Veronese morphism $\phi_s$. Thus $\nu_i >\epsilon_i$ for some $i=1,\ldots,M-1$, and therefore $\xx$ is $\fq$-Frobenius nonclassical. 
\end{proof}

As  mentioned in the introduction,  the construction of plane curves of degree $d<q/15$ attaining  the bound (\ref{SV-1}) requires constructing $\fq$-Frobenius nonclassical curves with respect to $\D_2$. Next, we take advantage of our previous characterization to find  explicit examples illustrating this phenomenon. 

Suppose that, in addition to  our standard  hypotheses, the curve  $\xx: F(x,y,z)=0$ satisfies the hypotheses of Corollary \ref{cor7}.  In particular,  $\xx$ is $\fq$-Frobenius nonclassical with respect to $\D_s$. Let $\cc:G(x,y,z)=0$ be the  curve of degree $s$, defined over $\F_{p^v}$, where
\begin{equation}\label{curva aux}
G(x,y,z)=\displaystyle\sum_{i+j+t = s}c_{ij}x^iy^jz^t.
\end{equation}
Note that $F(x,y,z)=G(x^n,y^n,z^n)$ and that the smoothness of  $\xx$ implies that $\cc$ is smooth as well.

\begin{thm}\label{exemplo}
If  $N_{p^v}(\cc)=s(s+p^v-1)/2$,  and there is no $\F_{p^v}$-rational point $P=(a:b:c:) \in \cc$ where $abc=0$, then 
$$N_q(\xx)=d(d+q-1)/2,$$
where $q=p^h$ and $d=sn$. In particular, if $s=2$ and $p^v>31$, then $\xx$ is a curve of degree $d<q/15$ attaining the bound (\ref{SV-1}). 
\end{thm}
\begin{proof}
Note that since $\xx$ is Frobenius nonclassical with respect to $\mathcal{D}_s$ and $s\geq 2$, 
Lemma \ref{lema5} implies that   $\xx$ is classical with respect to $\D_1$. Therefore, since $p>M=\dim \mathcal{D}_s$, Proposition \ref{fnc impl nc}  implies that  $\xx$ is $\fq$-Frobenius classical with respect to $\D_1$. Hence bound  (\ref{SV-1}) gives $N_q(\xx)\leq d(d+q-1)/2.$

Recall that $\xx: F(x,y,z)=0$ and $\cc: G(x,y,z)=0$ are such that  $F(x,y,z)=G(x^n,y^n,z^n)$ and $n=\frac{q-1}{p^v-1}$. Therefore, the map $\pi:\xx(\fq) \lra \cc(\F_{p^v})$, given by $\pi(\alpha:\beta:\gamma)\mapsto(\alpha^n:\beta^n:\gamma^n)$, is well defined. Since the norm function $x\mapsto x^{\frac{q-1}{p^v-1}}$ maps $\F_q$ onto $\F_{p^v}$, we have 
\begin{equation}\label{fibra}
\xx(\fq)=\bigcup\limits_{Q \in \cc(\F_{q^v})}\pi^{-1}(Q).
\end{equation}

For  $Q=(a:b:c) \in \cc$,  with $abc \neq 0$, we have  $\#\pi^{-1}(Q)=n^2$, and then 
$N_q(\xx)=n^2N_{p^v}(\cc).$ Therefore, 
$$
N_q(\xx)=\frac{n^2s(s+p^v-1)}{2}=\displaystyle\frac{s}{2}\cdot\left(\displaystyle\frac{(q-1)^2}{(p^v-1)^2}s+\displaystyle\frac{(q-1)^2}{(p^v-1)}\right)=\displaystyle\frac{sn(sn+q-1)}{2},
$$
and the result follows. Note that in the case  $s=2$ and $p^v>31$, the curve $\xx$ has degree 
$d=2n=\frac{2(q-1)}{p^v-1}<\frac{q-1}{15}<\frac{q}{15},$ as claimed.

\end{proof}

Constructing  curves illustrating  case $s=2$ in Theorem \ref{exemplo} is straightforward. One  need only select one of the many irreducible conics $\cc$, defined over $\F_{p^v}$, with no $\F_{p^v}$-rational points $P:=(a:b:c)$ with $abc=0$. Since $N_{p^v}(\cc)=p^v+1$, the curve $\cc$ attains bound (\ref{SV-1}), and the result follows. 


\section{The case $s=2$}\label{The case $s=2$}

As mentioned in Section \ref{The general case}, if $s=1$, the curve 
$$
\xx:\displaystyle\sum_{ i+j+t = s}c_{ij}x^{in}y^{jn}z^{tn}=0
$$ 
is a Fermat curve $ax^n+by^n=z^n$, and its classicality and $\fq$-Frobenius classicality with respect to $\D_1$ and $\D_2$ were studied in \cite{GV1} and \cite{GV2}, respectively. In this section, we exploit the case $s=2$. More precisely, we consider the curve $$\xx:F(x,y,z)=0,$$ where
\begin{equation}\label{xs2}
F(x,y,z)=a_1x^{2n}+a_2x^ny^n+a_3y^{2n}+a_4x^nz^n+a_5y^nz^n+a_6z^{2n},
\end{equation}
with  $a_i \in \fq$, $i \in \{1,2,3,4,5,6\}$, and assume the following:
\begin{enumerate}
\item[(4.i)] $p>2$
\item[(4.ii)]  $\xx$ is nonsingular (in particular,  $a_1a_3a_6 \neq 0$)
\item[(4.iii)] At least one of the coefficients $a_2$, $a_4$ and $a_5$ is nonzero. In other words, equation (\ref{xs2}) is not of Fermat type.

\end{enumerate}

With these assumptions, we  prove that $\xx$ is $\fq$-Frobenius classical with respect to $\D_1$ and   establish necessary and sufficient conditions for the curve $\xx$ to be $\fq$-Frobenius nonclassical with respect to $\D_2$.

\begin{rem}\label{conic irred}
Since $\xx$ is irreducible, the conic given by the equation $a_1x^{2}+a_2xy+a_3y^{2}+a_4xz+a_5yz+a_6z^{2}=0$ is irreducible, i.e.,
\begin{equation}\label{det conic}
 \left|\begin{array}{ccc}
a_1 & a_2/2 & a_4/2 \\
a_2/2 & a_3 & a_5/2 \\
a_4/2 & a_5/2 & a_6 \end{array} \right|\neq 0.
\end{equation}
\end{rem}
 
Throughout this section, $F(x,y,1)$ will be denoted by $f(x,y)$.

\begin{prop}\label{inflex}
There exists a point $P \in \xx$ whose $(\D_1,P)$-order sequence is $(0,1,n)$. In particular, if $\xx$ is nonclassical with respect to $\D_1$, then $p|n(n-1)$. 
\end{prop}
\begin{proof}
Using  assumption (4.iii), without loss of generality, we assume $a_2 \neq 0$. If $P=(u:0:1) \in \xx$, then  $f(u,0)=a_1u^{2n}+a_4u^n+a_6=0$ (in particular, $u \neq 0$) and the tangent line to $\xx$ at $P$ is given by $\ell_P:x-uz=0$. Thus
\begin{equation}\label{tgt}
 f(u,y)=y^ng(y),
\end{equation}
where $g(y)=a_2u^n+a_5+a_3y^n\neq 0$. Then $I(P, \ell_P \cap \xx)=n$ if  and only if $a_2u^n+a_5 \neq 0$. Our remaining problem reduces to find a point $P=(u:0:1) \in \xx$ for which $a_2u^n+a_5 \neq 0$.

Suppose  there is no such point. That is,  all the roots of $a_1x^{2n}+a_4x^{n}+a_6=0$ are roots of $a_2x^n+a_5 = 0$. This implies that $a_1x^2+a_4x+a_6=0$ has a double root $\alpha=-a_5/a_2$, which yields
\begin{equation}\label{a2 dif 0}
a_4^2-4a_1a_6=0 \ \ \textrm{ and } \ \ \ a_1a_5^2-a_2a_4a_5+a_2^2a_6=0.
\end{equation}
One can easily  check that (\ref{a2 dif 0})  gives

$$
 \left|\begin{array}{ccc}
a_1 & a_2/2 & a_4/2 \\
a_2/2 & a_3 & a_5/2 \\
a_4/2 & a_5/2 & a_6 \end{array} \right|= 0,
$$
which contradicts (\ref{det conic}). 

The last statement of the proposition follows directly from Proposition \ref{SV1.7}.
\end{proof}

\begin{prop}\label{prop8}
The curve $\xx$ is classical with respect to $\D_1$. Consequently, $\xx$ is $\fq$-Frobenius classical with respect to $\D_1$.
\end{prop}
\begin{proof}
Suppose that $\xx$ is nonclassical with respect to $\D_1$. Since $\xx$ is nonsingular and $p>2$, by \cite[Corollary 2.2]{Pa},  $p|2n-1$. On the other hand, by Proposition \ref{inflex}, $p|n(n-1)$. However,  $\gcd(2n-1,n^2-n)=1$, and then $\xx$ must be classical with respect to $\D_1$.  Thus by Proposition \ref{fnc impl nc},  the curve $\xx$ is $\fq$-Frobenius classical with respect to $\D_1$.
\end{proof}

\begin{rem}
It follows from Proposition \ref{prop8} that the bound (\ref{SV-1}) can always be applied to the curve $\xx$. In other words, $N_q(\xx) \leq d(d+q-1)/2$.
\end{rem}

	We now study the (non)classicality and $\fq$-Frobenius (non)classicality of $\xx$ with respect to the linear series $\D_2$, making the following assumptions:
\begin{itemize}
\item[(4.iv)] $p>7$
\item[(4.v)] $n>2$.
\end{itemize}

The following theorems will be proved after a sequence of partial results.

\begin{thm}\label{teo12}
The curve $\xx$ is nonclassical with respect to $\D_2$ if and only if one of the following holds:
\begin{enumerate}
\item [(1)] $p|n-1$
\item [(2)] $p|2n-1$ and all but one of the coefficients $a_2$, $a_4$ and $a_5$ are zero.
\end{enumerate}
\end{thm}

\begin{thm}\label{main frob}
The curve $\xx$ is $\fq$-Frobenius nonclassical with respect to $\D_2$ if and only if one of the following holds:
\begin{enumerate}
\item [(1)] $p|n-1$ and $n=\frac{p^h-1}{p^v-1}$ for some integer $v<h$ with $v|h$, and $\xx$ is defined over $\F_{p^v}$
\item [(2)] $p|2n-1$, all but one of the coefficients $a_2$, $a_4$ and $a_5$ are zero, $n=\frac{p^h-1}{2(p^v-1)}$ for some integer $v<h$ with $v|h$ and, up to an $\fq$-scaling of the coordinates, the curve $\xx$ is defined over $\F_{p^v}$.
\end{enumerate}
\end{thm}


The next three lemmas will provide the key results for the proof of Theorem \ref{teo12}.

\begin{lem}\label{lema10}
If $p|(n+1)(n-2)$, then $\xx$ is classical with respect to $\D_2$.
\end{lem}
\begin{proof}
Since $\xx$ is classical with respect to $\D_1$, the $\D_2$-order sequence of $\xx$ is given by $(0,1,2,3,4,\epsilon)$, where $\epsilon \geq 5$. Suppose that $\epsilon>5$, i.e., $\xx$ is nonclassical for $\D_2$. Then by \cite[Proposition 2]{GV1},  $\epsilon=p^s$, for some $s>0$.

First, assume  $p|n-2$. Hence $n=mp^v+2$, for some $m,v>0$, with gcd($m,p$)=$1$ and then $f(x,y)=0$ can be written as
\begin{equation}
a_1(x^{2m})^{p^v}x^4+a_2(x^my^m)^{p^v}x^2y^2+a_3(y^{2m})^{p^v}y^4+a_4(x^m)^{p^v}x^2+a_5(y^m)^{p^v}y^2+a_6=0. \nonumber
\end{equation}

Let $P=(u:w:1) \in \xx$ with $uw \neq 0$ and consider the projective closure $\qq_P \subset \p^2(\kk)$ of the curve given by
\begin{eqnarray}\nonumber
r(x,y)&=&a_1(u^{2m})^{p^\nu}x^4+a_2(u^mw^m)^{p^\nu}x^2y^2+a_3(w^{2m})^{p^\nu}y^4 \\
      &+&a_4(u^m)^{p^\nu}x^2+a_5(w^m)^{p^\nu}y^2+a_6=0.\nonumber
\end{eqnarray}
Note that $\qq_P$ is an irreducible  quartic. In fact, $\qq_P$ is projectively equivalent to the curve $\cc$ given by
$$
a_1x^4+a_2x^2y^2+a_3y^4+a_4x^2z^2+a_5y^2z^2+a_6z^4=0.
$$
 The curve $\cc$, on the other hand, is nonsingular. Indeed if $(a:b:c)$ is a singular point of $\cc$, then $(\alpha:\beta:\gamma)$ is a singular point of $\xx$, where $\alpha,\beta,\gamma \in \kk$ are roots of $x^n=a^2, x^n=b^2$, and $x^n=c^2$ respectively. This contradicts the smoothness of $\xx$. 

Now for all $P=(u:w:1) \in \xx$ with $uw \neq 0$, 
\begin{eqnarray}
r(x,y) & = & r(x,y)-f(x,y) \nonumber \\
& = & a_1(u^{2m}-x^{2m})^{p^v}x^4+a_2(u^mw^m-x^my^m)^{p^v}x^2y^2 \nonumber \\
& + &a_3(w^{2m}-y^{2m})^{p^v}y^4+ a_4(u^m-x^m)^{p^v}x^2+a_5(w^m-y^m)^{p^v}y^2. \nonumber
\end{eqnarray}
 Then $I(P,\qq_P \cap \xx) \geq p^v$. Let $\hh_P^2$ be the osculating conic to $\xx$ at $P$. Since $\epsilon=p^s$, we have that $I(P,\hh_P^2 \cap \xx) \geq p^s$. However,  Lemma \ref{lema4}   with our assumption  that $p>7$ gives $$I(P,\hh_P^2 \cap \qq_P) \geq p \geq 11>8=deg(\hh_P^2)\cdot deg(\qq_P),$$ which implies, by  B\'ezout's Theorem, that $\hh_P^2$ is a component of $\qq_P$.  This contradicts  the irreducibility of $\qq_P$. Therefore, $\xx$ is classical.

Suppose $p|n+1$, and let $m,v>0$ be such that $n=mp^v-1$ and $\gcd(m,p)=1$. From $f(x,y)=0$ we obtain
\begin{eqnarray}
0 & = & f(x,y)x^2y^2 \Lra \nonumber \\
0 & = & a_1(x^{2m})^{p^v}y^2+a_2(x^my^m)^{p^v}xy+a_3(y^{2m})^{p^v}x^2 \nonumber \\
  & + &a_4(x^m)^{p^v}xy^2+a_5(y^m)^{p^v}x^2y+a_6x^2y^2. \nonumber
\end{eqnarray}

Consider a point $P=(u:w:1) \in \xx$ with $uw \neq 0$ and the projective closure $\qq^{'}_P \subset \p^2(\kk)$ of the curve given by $l(x,y)=0$, where
\begin{eqnarray}
l(x,y)&=&a_6x^2y^2+a_5(w^m)^{p^v}x^2y+a_4(u^m)^{p^v}xy^2 \nonumber \\
      &+&a_3(w^{2m})^{p^v}x^2+a_2(u^mw^m)^{p^v}xy+a_1(u^{2m})^{p^v}y^2. \nonumber
\end{eqnarray}
Since $a_6 \neq 0$,  $\qq^{'}_P$ is a quartic. Let $\alpha=u^{mp^v}$ and $\beta=w^{mp^v}$. Multiplying $l(x,y)$ by $1/\alpha^2 \beta^2$, we see that $\qq_P^{'}$ is the projective closure of the curve given by the equation
$$
a_6\frac{x^2y^2}{\alpha^2 \beta^2}+a_5\frac{x^2y}{\alpha^2 \beta}+a_4\frac{xy^2}{\alpha \beta^2}+a_3\frac{x^2}{\alpha^2}+a_2\frac{xy}{\alpha \beta}+a_1\frac{y^2}{\beta^2}=0.
$$
Hence $\qq^{'}_P$ is projectively equivalent to the curve $\mathcal{Y}$ given by
$$
H(x,y,z)=a_6x^2y^2+a_5x^2yz+a_4xy^2z+a_3x^2z^2+a_2xyz^2+a_1y^2z^2=0.
$$
Then Lemma \ref{quart irr} and Remark \ref{conic irred} imply that $\qq^{'}_P$ is irreducible.

Moreover,
\begin{equation}\nonumber
l(x,y)  =  l(x,y)-f(x,y)x^2y^2.
\end{equation}

Therefore, $I(P,\qq^{'}_P \cap \xx) \geq p^v\geq 11$. If $\hh_P^{2}$ is the osculating conic to $\xx$ at $P$, we have $I(P,\hh_P^{2} \cap \xx) \geq p^s\geq 11$. By Lemma \ref{lema4} and B\'ezout's Theorem, $\hh_P^{2}$ is a component of $\qq^{'}_P$.This is a contradiction, and thus the curve $\xx$ is classical.
\end{proof}

\begin{lem}\label{lema9}
If $\xx$ is nonclassical with respect to $\D_2$, then $p|(n-1)(2n-1)$.
\end{lem}
\begin{proof}
By Proposition \ref{inflex}, there exists a point $P \in \xx$ with order sequence $(0,1,n)$ with respect to $\D_1$, i.e., $0$, $1$ and $n$ are  the possible intersection multiplicities of $\xx$ with a line at $P$. Hence there are degenerated conics in $\p^2(\kk)$ whose intersection multiplicities with $\xx$ at $P$ are $0,1,2,n,n+1$ and $2n$. Since $\D_2$ has projective dimension $5$, these are  the possible intersection multiplicities of $\xx$ with a conic at $P$. In other words, the order sequence of $P$ with respect to $\D_2$ is $(0,1,2,n,n+1,2n)$. Thus by Proposition \ref{SV1.7}, $p$ divides   $n(n-1)(2n-1)(n+1)(n-2)$. Since the irreducibility of $\xx$ together with Lemma \ref{lema10} gives $p\nmid n(n+1)(n-2)$, the result follows.
\end{proof}

The next two lemmas will address the converse of Lemma \ref{lema9}.

\begin{lem}\label{lema11}
If $p|n-1$, then $\xx$ is nonclassical with respect to $\D_2$.
\end{lem}
\begin{proof}
It follows immediately from Lemma \ref{lema5} applied to the case $s=2$.
\end{proof}

\begin{lem}\label{teo 14}
Assume that $p|2n-1$. The curve $\xx$ is nonclassical with respect to $\D_2$ if and only if all but one of the coefficients $a_2$, $a_4$ and $a_5$ are zero.
\end{lem}
\begin{proof}
Let $m,v$ be such that $2n=mp^v+1$ and $\gcd(m,p)=1$. Assume that all but one of the coefficients $a_2$, $a_4$ and $a_5$ are zero. We may suppose that $F(x,y,z)=a_1x^{2n}+a_2x^ny^n+a_3y^{2n}+a_6z^{2n}$ with $a_2 \neq 0$ (the other two cases are analogous). We have 
\begin{eqnarray}\label{ld}
0&=& a_1x^{2n}+a_2x^ny^n+a_3y^{2n}+a_6\Lra \nonumber\\
-a_2x^ny^n&=&a_1(x^m)^{p^v}x+a_3(y^m)^{p^v}y+a_6 \Lra \nonumber \\
(a_2^{2/p^v} x^my^m)^{p^v}xy&=&((a_1^{1/p^v}x^m)^{p^v}x+(a_3^{1/p^v}y^m)^{p^v}y+(a_6^{1/p^v})^{p^v})^2. 
\end{eqnarray}
Since $\xx$ is classical with respect to $\D_1$, the $\D_2$-order sequence of $\xx$ is $(0,1,2,3,4,\epsilon)$ for some $\epsilon \geq 5$. In view of (\ref{ld}), Theorem \ref{GV1-1} implies  $\epsilon \geq p^v>5$. Hence $\xx$ is nonclassical for $\D_2$.

Now assume  $\xx$ is nonclassical and suppose that at least two of the constants $a_2,a_4$ and $a_5$ are nonzero. Recall that the smoothness of $\xx$ implies $a_1a_3a_6\neq 0$, and then after scaling, we may set $a_1=a_3=a_6=1$. Thus since
$f(x,y)=x^{2n}+a_2x^ny^n+y^{2n}+a_4x^n+a_5y^n+1=0 \in \kk(\xx)$, we have that
$$(x^{2n}+a_2x^ny^n+y^{2n}+a_4x^n+a_5y^n+1)(x^{2n}-a_2x^ny^n+y^{2n}-a_4x^n+a_5y^n+1)=0,$$
and then
\begin{eqnarray}\label{eq1}
& &x^{4n} + (2-a_2^2 )x^{2n}y^{2n}  + (2-a_4^2 )x^{2n} +  y^{4n} + (a_5^2 + 2)y^{2n} + 1 \\
&=&2y^n((a_2a_4-a_5)x^{2n} -a_5y^{2n} -a_5).\nonumber
\end{eqnarray}
Squaring both sides of (\ref{eq1}) yields

\begin{eqnarray}\label{eq2}
&&\Big((x^{2m})^{p^v}x^2 + (2-a_2^2 )(x^{m}y^{m})^{p^v}xy  + (2-a_4^2 )(x^{m})^{p^v}x  \nonumber\\
&+&  (y^{2m})^{p^v}y^2 + (a_5^2 + 2)(y^{m})^{p^v}y + 1\Big)^2 \\
&=&4(y^{m})^{p^v}y\Big((a_2a_4-a_5)(x^{m})^{p^v}x -a_5(y^{m})^{p^v}y -a_5\Big)^2.\nonumber
\end{eqnarray}

Let $P=(u:w:1) \in \xx$ with $uw \neq 0$ and $\qq_P$ be the projective closure of the quartic given by $r(x,y)=0$, where
\begin{eqnarray}\label{eq5}
r(x,y)&=&\Big((u^{2m})^{p^v}x^2 + (2-a_2^2 )(u^{m}w^{m})^{p^v}xy  + (2-a_4^2 )(u^{m})^{p^v}x  \nonumber \\
      &+&  (w^{2m})^{p^v}y^2 + (a_5^2 + 2)(w^{m})^{p^v}y + 1\Big)^2 \\
&-&4(w^{m})^{p^v}y\Big((a_2a_4-a_5)(u^{m})^{p^v}x -a_5(w^{m})^{p^v}y -a_5\Big)^2.\nonumber
\end{eqnarray}

We claim that $\qq_P$ is irreducible. In fact,  via $(x:y:z)\mapsto (u^{mp^v}x:w^{mp^v}y:z)$,  the quartic $\qq_p$ is projectively equivalent to 
$$
\Big((x +  y+  z)^2-a_2^2xy  -a_4^2 xz  + a_5^2yz \Big)^2-4\Big((a_2a_4 - a_5)x - a_5y - a_5z\Big)^2yz=0.
$$
 Thus if $\qq_P$ is reducible, then Theorem \ref{main ap} implies that $a_2^2 + a_4^2 + a_5^2 - a_2a_4a_5=4$ (since we are assuming that at least two of the constants $a_2,a_4$ and $a_5$ are nonzero). But then 

$$
\left|
\begin{array}{ccc}
1 & a_2/2 & a_4/2 \\
a_2/2 & 1 & a_5/2 \\
a_4/2 & a_5/2 & 1
\end{array}
\right|=\frac{a_2a_4a_5-(a_2^2+a_4^2+a_5^2)}{4}+1=0,
$$
which is a contradiction to (\ref{det conic}).

Hence using the same arguments in  the proof of Lemma \ref{lema10}, we get $I(P,\qq_P \cap \xx) \geq p$. Since $\xx$ is classical with respect to $\D_1$ and nonclassical with respect to $\D_2$, by \cite[Proposition 2]{GV1}  the order sequence of $\xx$ with respect to $\D_2$ is $(0,1,2,3,4,p^s)$ for some $s>0$. Therefore, if $\hh_P^2$ is the osculating conic to $\xx$ at $P$, we have $I(P,\hh^2_P \cap \xx) \geq p^s$. Using Lemma \ref{lema4}, as in the previous cases, we obtain a contradiction by B\'ezouts Theorem since we are assuming that $p>7$. 
\end{proof}
\noindent
{\bf Proof of Theorem \ref{teo12}}
It follows directly   from Lemmas \ref{lema9},  \ref{lema11}, and \ref{teo 14}. \qed

We use the following lemmas to build our proof of Theorem \ref{main frob}.

\begin{lem}\label{teo13}
Assume that $p|n-1$. Then $\xx$ is $\fq$-Frobenius nonclassical with respect to $\D_2$ if  and only if $n=\frac{p^h-1}{p^v-1}$, with $h>v$, $v|h$ and $\xx$ is defined over $\F_{p^v}$.
\end{lem}
\begin{proof}
If $n=\frac{p^h-1}{p^v-1}$, with $h>v$, $v|h$ and $\xx$ is defined over $\F_{p^v}$, by Corollary \ref{cor7} applied in the case $s=2$, $\xx$ is $\fq$-Frobenius nonclassical with respect to $\D_2$. For the converse, note that by Proposition \ref{fnc impl nc},  $\xx$ must be  nonclassical with respect to $\D_2$. Since $\xx$ is classical with respect to $\D_1$ (Proposition \ref{prop8}), its $\D_2$-order sequence is $(0,1,2,3,4,\epsilon)$, where $\epsilon>5$. The $\fq$-Frobenius nonclassicality of $\xx$ with respect to $\D_2$ is equivalent to

\begin{displaymath}
\left|
  \begin{array}{cccccc}
  1 & x^q &y^q&x^{2q}& x^qy^q & y^{2q} \\
  1 & x & y& x^2& xy& y^2 \\
  0 & D_\tau^{(1)}(x) & D_\tau^{(1)}(y)& D_\tau^{(1)}(x^2)&D_\tau^{(1)}(xy)&D_\tau^{(1)}(y^2)\\
  0 & D_\tau^{(2)}(x) & D_\tau^{(2)}(y)& D_\tau^{(2)}(x^2)&D_\tau^{(2)}(xy)&D_\tau^{(2)}(y^2)\\
  0 & D_\tau^{(3)}(x) & D_\tau^{(3)}(y)& D_\tau^{(3)}(x^2)&D_\tau^{(3)}(xy)&D_\tau^{(3)}(y^2)\\
  0 & D_\tau^{(4)}(x) & D_\tau^{(4)}(y)& D_\tau^{(4)}(x^2)&D_\tau^{(4)}(xy)&D_\tau^{(4)}(y^2)
  \end{array}
  \right| = 0,
  \end{displaymath}
where $\tau$ is a separating variable of $\fq(\xx)$. Then by \cite[Corollary 1.3]{SV}  $\Phi_q(P) \in \hh_P^{2}$ for infinitely many points of $\xx$. Hence the result follows from Theorem \ref{teo6}.
\end{proof}

The next lemma follows from \cite[Theorem 3.2]{FM}.

\begin{lem}\label{pol}
Let $K$ be an arbitrary field. Consider nonconstant polynomials $b_1(x), b_2(x) \in K[x]$, and let $l$ and $m$ be positive integers. Then $$y^l-b_1(x) \text{  divides  } y^m-b_2(x)$$ if  and only if $l|m$ and $b_2(x)=b_1(x)^{\frac{m}{l}}$.
\end{lem}

\begin{lem}\label{teo 15}
Assume that $p|2n-1$. The curve $\xx$ is $\fq$-Frobenius nonclassical with respect to $\D_2$ if and only if all but one of the coefficients $a_2$, $a_4$ and $a_5$ are zero, $n=\frac{q-1}{2(p^v-1)}$ for some integer $v<h$ with $v|h$, and up to an $\fq$-scaling of the coordinates, the curve $\xx$ is defined over $\F_{p^v}$.
\end{lem}
\begin{proof},
Suppose that $\xx$ is $\fq$-Frobenius nonclassical. By Proposition \ref{fnc impl nc}, the curve $\xx$ is nonclassical and therefore, by Theorem \ref{teo 14},  all but one of the coefficients $a_2$, $a_4$, and $a_5$ are zero. We can assume, without loss of generality, that $a_4 \neq 0$. Dehomogenizing $F(x,y,z)$ with respect to $z$ and setting $a:=-a_1/a_3$, $b:=-a_4/a_3$, and $c:=-a_6/a_3$, we obtain that $\xx$ is given by the affine equation
\begin{equation}\label{eq nova}
y^{2n}=ax^{2n}+bx^n+c.
\end{equation}
Since $p \nmid 2n$, we have that $x$ is a separating variable of $\fq(\xx)$. The assumption that $\xx$ is $\fq$-Frobenius nonclassical is equivalent to $W=0 \in \fq(\xx)$, where

\begin{equation}\label{W}
W:=\left|\begin{array}{ccccc}
x-x^{q} & x^{2}-x^{2q} & y-y^{q} & xy-x^{q}y^{q} & y^{2}-y^{2q} \\
1 & 2x & D_x^{(1)}(y) & D_x^{(1)}(xy) & D_x^{(1)}(y^{2}) \\
0 & 1 & D_x^{(2)}(y) & D_x^{(2)}(xy) & D_x^{(2)}(y^{2}) \\
0 & 0 & D_x^{(3)}(y) & D_x^{(3)}(xy) & D_x^{(3)}(y^{2}) \\
0 & 0 & D_x^{(4)}(y) & D_x^{(4)}(xy) & D_x^{(4)}(y^{2}) 
\end{array}\right|.
\end{equation}
Using the formula $D_x^{(i)}(fg)=\sum\limits_{j=0}^{i}D_x^{(j)}(f)D_x^{(i-j)}(g)$ (see e.g. \cite[Lemma 5.72]{HKT}) and elementary properties of Determinants, we obtain

\begin{equation}\nonumber
W=\left|\begin{array}{ccccc}
x-x^{q} & x^{2}-x^{2q} & y-y^{q} & 0 & -(y^{q}-y)^2 \\
1 & 2x & D_x^{(1)}(y) & y-y^q & 0 \\
0 & 1 & D_x^{(2)}(y) & D_x^{(1)}(y) & (D_x^{(1)}(y))^2 \\
0 & 0 & D_x^{(3)}(y) & D_x^{(2)}(y) & 2D_x^{(1)}(y)D_x^{(2)}(y) \\
0 & 0 & D_x^{(4)}(y) & D_x^{(3)}(y) & 2D_x^{(1)}(y)D_x^{(3)}(y)+ (D_x^{(2)}(y))^2
\end{array}\right|.
\end{equation}

Equation (\ref{eq nova})  with the hypothesis $p|2n-1$ gives us

\begin{equation}\nonumber
D_x^{(1)}(y)=\frac{2ax^{2n-1}+bx^{n-1}}{2y^{2n-1}} \ \ \textrm { and } \ \ \ D_x^{(i)}(y)=\frac{(n-1)\ldots(n-i+1)bx^{n-i}}{2i!y^{2n-1}}
\end{equation}
for $i>1$. Through standard computations and bearing in mind that $p|2n-1$, we obtain 

\begin{eqnarray}
W & = &\frac{b^2x^{2n-6}}{1024y^{8n-4}}  \big(-2bx^{n}y^{2n}-2y^{4n+q-1}-2abx^{3n+q-1}+ 2abx^{3n}+y^{4n}\nonumber \\
  & + & 2bx^{n}y^{2n+q-1}+ y^{4n+2q-2}+ a^2x^{4n}+b^2x^{2n}+a^2x^{4n+2q-2}-b^2x^{2n+q-1}\nonumber\\ 
  & - & 2a^2x^{4n+q-1}- 2ax^{2n}y^{2n}+2ax^{2n}y^{2n+q-1}+2ax^{2n+q-1}y^{2n}\nonumber\\ 
  & - & 2ax^{2n+q-1}y^{2n+q-1}\big).\nonumber 
\end{eqnarray}

Therefore, $W=\frac{b^2x^{2n-6}}{1024y^{8n-4}}\cdot W_1 \cdot W_2$, where
\begin{equation}\nonumber
W_1:=ax^{2n+q-1}-y^{2n+q-1}+bx^{\frac{2n+q-1}{2}}+y^{2n}-ax^{2n}-bx^n
\end{equation}
and
\begin{equation}\nonumber
W_2:=ax^{2n+q-1}-y^{2n+q-1}-bx^{\frac{2n+q-1}{2}}+y^{2n}-ax^{2n}-bx^n.
\end{equation}

From equation (\ref{eq nova}), we can write

\begin{equation}\label{mod f1}
W_1=y^{2n+q-1}-ax^{2n+q-1}-bx^{\frac{2n+q-1}{2}}-c 
\end{equation}
and
\begin{equation}\label{mod f2}
W_2=y^{2n+q-1}-ax^{2n+q-1}+bx^{\frac{2n+q-1}{2}}-c.
\end{equation}

Now consider $W_1$ and $W_2$ as polynomials. Since $W=0 \in \fq(\xx)$, there are two possibilities:
\begin{itemize}
\item[(i)] $(y^{2n}-ax^{2n}-bx^n-c) \big| W_1$. In this case, by Lemma \ref{pol},  $2n|2n+q-1$ and
$$
ax^{2n+q-1}+bx^{\frac{2n+q-1}{2}}+c = (ax^{2n}+bx^n+c)^{\frac{2n+q-1}{2n}}.
$$
It can be checked that the equality above implies $\frac{2n+q-1}{2n}=p^v$ for some $v>0$, i.e., $n=\frac{q-1}{2(p^v-1)}$, and hence $v$ is a proper divisor of $h$. Furthermore, $a^{p^v}=a$, $b^{p^v}=b$ and $c^{p^v}=c$, which means that $a,b,c \in \F_{p^v}$. 
\item[(ii)] $(y^{2n}-ax^{2n}-bx^n-c) \big| W_2$. By Lemma \ref{pol}, $n=\frac{q-1}{2(p^v-1)}$, where $v$ is a proper divisor of $h$. Moreover, $a^{p^v}=a$, $b^{p^v}=-b$ and $c^{p^v}=c$. Hence $a, c \in \F_{p^v}$ and $b \in \fq$ is such that $b^{p^v-1}=-1$. Since $b^2 \in \F_{p^v}$, there exists $\alpha \in \fq$ such that $\alpha^{2n}=b^2$, using the surjectivity of the norm map $N:\fq \lra \F_{p^v}$. Thus up to the $\fq$-scaling $(x,y)\mapsto(\alpha x,y)$, the curve $\xx$ is defined over $\F_{p^v}$.
\end{itemize}

Conversely, assume that all but one of the coefficients $a_2$, $a_4$ and $a_5$ are zero, $n=\frac{q-1}{2(p^v-1)}$ for some integer $v<h$ with $v|h$ and that, up to $\fq$-scaling, the curve $\xx$ is defined over $\F_{p^v}$. We can suppose, without loss of generality, that $a_4 \neq 0$ and that $a_1,a_3,a_4,a_6 \in \F_{p^v}$. Then the curve $\xx$ is determined by the affine equation (\ref{eq nova}) with $a,b,c \in \F_{p^v}$. Hence$$W=\frac{b^2x^{2n-6}}{1024y^{8n-4}}\cdot W_1 \cdot W_2,$$ with $W, W_1$ and $W_2$ as in (\ref{W}), (\ref{mod f1}), and (\ref{mod f2}), respectively. Since $n=\frac{q-1}{2(p^v-1)}$, we have 
$$
2n+q-1=2np^v.
$$
Therefore,
$$
W_1=y^{2n+q-1}-ax^{2n+q-1}-bx^{\frac{2n+q-1}{2}}-c=(y^{2n}-ax^{2n}-bx^{n}-c)^{p^v}=0.
$$
Thus $W=0$, i.e., $\xx$ is $\fq$-Frobenius nonclassical with respect to $\D_2$. 
\end{proof}

\noindent
{\bf Proof of Theorem \ref{main frob}} It follows directly  from Lemmas \ref{lema9}, \ref{teo13} and \ref{teo 15}. \qed

\section{The number of rational points}
In this section, we use the preceding results  to discuss the possible values of  $N_q(\xx) $  in the case $s=2$. Since  the necessary and sufficient conditions for   the  $\fq$-Frobenius nonclassicality of $\xx$ were established,  we will be able to provide the exact number of $\fq$-rational points for these curves. In  the remaining cases, i.e., for the $\fq$-Frobenius classical curves $\xx$,  the St\"orh-Voloch bound (\ref{SV-s2}) gives
\begin{equation}\label{sv con}
N_q(\xx) \leq \frac{2d(5d+q-10)}{5}.
\end{equation}
where $d=\deg \xx$.
The next result gives the  number of $\fq$-rational points on the $\fq$-Frobenius nonclassical curves $\xx$ satisfying condition \emph{(1)} of Theorem \ref{main frob}.

\begin{thm}\label{nr ponto}
If $n=\frac{q-1}{p^v-1}$, with $v<h$ such that $v|h$, and $a_1,\cdots,a_6\in \F_{p^v}$ are such that the curve 
$\xx: a_1x^{2n}+a_2x^{n}y^{n}+a_3y^{2n}+a_4x^{n}z^{n}+a_5y^{n}z^{n}+a_6z^{2n}=0$
 is smooth, then 
\begin{equation}\label{nr pontos}
N_q(\xx)=n\Big(n(p^v+1)-\delta(n-1)\Big),
\end{equation}
where $\delta$ is the number of $\F_{p^v}$-rational points $P=(a:b:c)$ on  the conic $\cc:a_1x^2+a_2xy+a_3y^{2}+a_4xz+a_5yz+a_6z^2=0$, satisfying $abc=0$.
\end{thm}
\begin{proof}
 As in the proof of Theorem \ref{exemplo}, consider the map $\pi:\xx(\fq) \lra \cc(\F_{p^v})$ given by $\pi(\alpha:\beta:\gamma)=(\alpha^n:\beta^n:\gamma^n)$. Since $\xx$ is nonsingular, $(1:0:0), (0:1:0), (0:0:1) \not\in \xx$. Hence $\# \pi^{-1}(Q)=n$ for all $Q=(a:b:c) \in \cc(\F_{p^v})$ such that $abc=0$. Additionally, $\# \pi^{-1}(Q)=n^2$ for all $Q=(a:b:c) \in \cc(\F_{p^v})$ such that $abc\neq 0$. Since $N_{p^v}(\cc)=p^v+1$, equation (\ref{fibra}) gives the result.
\end{proof}

\begin{exa}\label{exe1}
Consider the curve $\xx:x^{88}+3x^{44}y^{44}+y^{88}+3x^{44}z^{44}+3y^{44}z^{44}+z^{88}=0$ over $\F_{43^2}$. The curve $\xx$ has degree $d=2n$, where $n=\frac{43^2-1}{43-1}$. It can be checked that the conic $\cc:x^{2}+3xy+y^2+3xz+3yz+z^2=0$ has no $\F_{43}$-rational points $P=(a:b:c)$ with $abc=0$. Hence (\ref{nr pontos}) gives  $N_q(\xx)=85184$. 
\end{exa}

For the curves corresponding to case \emph{(2)} of Theorem \ref{main frob}, we have the following.

\begin{thm}\label{nr pontos2}
If $n=\frac{q-1}{2(p^v-1)}$, with $v<h$ such that $v|h$, and $a,b,c\in \F_{p^v}^{*}$  are such that the curve 
$\xx: ax^{2n}+bx^ny^n+cy^{2n}+z^{2n}=0$ is smooth, then
\begin{equation} \label{pontos}
N_q(\xx)=n\Big(q+3-(2n-1)\cdot \eta\Big),
\end{equation}
where $\eta$ is the number of distinct $\F_{p^v}$-roots of $ax^2+bx+c=0$.
\end{thm}
\begin{proof}
Considering the irreducible conic $\mathcal{C}: ax^{2}+bxy+cy^{2}+z^{2}=0$, the map $\varphi: \xx \longrightarrow  \mathcal{C}$, given by
$(x:y:z) \mapsto (x^n:y^n:z^n)$ is well defined. Thus since $n=\frac{q-1}{2(p^v-1)}$, a point $P \in \mathcal{X} $ is $\F_{q}$-rational if and only if the nonzero coordinates of $Q=\varphi(P)$ satisfy the equation $t^{2(p^v-1)}=1$. That is, the point $Q$ is defined over either $\F_{p^v}$ or $\lambda \cdot\F_{p^v}$, where $\lambda$ is such that $\lambda^{p^v-1}=-1$.  Note  that the fiber of each point $Q=(x:y:z) \in \mathcal{C}$ has either $n^2$ or $n$ points, with  the latter case corresponding to the points for which $xyz=0$. Therefore, counting the  $\F_{q}$-rational points on $\xx$  reduces  to counting the  points $Q=(x:y:z) \in \mathcal{C}$ defined over the set $S:= \lambda \cdot\F_{p^v} \cup\F_{p^v} $, where  $\lambda^{p^v-1}=-1$.

The  computation will be based on two types of points $(x:y:z) \in \mathcal{C}$.
\begin{enumerate}
\item[(i)] {\it Case $xyz\neq 0$}.  For  $f(x,y):=ax^{2}+bxy+cy^{2}+1=0$, let $x_0,y_0 \in S\backslash \{0\}$ be such that $f(x_0,y_0)=0$. Since $a,b,c\in \F_{p^v}$,  either $x_0,y_0 \in \F_{p^v}$ or $x_0,y_0  \in \lambda \cdot \F_{p^v}$. Hence the number  sought  is given by the number  of points $(x_0,y_0)\in \F_{p^v}^{*}\times \F_{p^v}^{*}$ on the union of the two  distinct and irreducible conics 
$$\mathcal{C}_1: ax^{2}+bxy+cy^{2}+1=0,$$
  {\center and  } 
$$\mathcal{C}_2:  ax^{2}+bxy+cy^{2}+1/\lambda^2=0.$$
Clearly this number is $2(p^v+1)-(\#\mathcal{Z}_1+\#\mathcal{Z}_2)$, where $\mathcal{Z}_i$ is the set of points $Q=(x:y:z)$, with $xyz=0$, on the projective closure of $\mathcal{C}_i$, $i=1,2.$ Let $\mathcal{Z}_i \cap \{z=0\}\subseteq \mathcal{Z}_i  $ be the set of points on the line $z=0$. Note that  $\mathcal{Z}_1 \cap \{z=0\}=\mathcal{Z}_2 \cap \{z=0\}=\mathcal{Z}_1 \cap\mathcal{Z}_2$, and then
$$\eta:=\#(\mathcal{Z}_1 \cap\mathcal{Z}_2)$$ 
is the number of distinct $\F_{p^v}$-roots of $ax^{2}+bx+c=0.$ Since $1/\lambda^2\in \F_{p^v}$ is not a square, we can see that
$$\#\Big((\mathcal{Z}_1 \cup \mathcal{Z}_2)\cap  \{xy=0\}\Big)=4,$$
and then  $\#\mathcal{Z}_1+\#\mathcal{Z}_2=4+2\eta$. Therefore, the number of $\F_{q}$-rational points on $\mathcal{X}$, with nonzero coordinates, is given by 
\begin{equation} \label{afins}
n^2\Big(2(p^v+1)-(4+2\eta)\Big).
\end{equation}

\item[(ii)] {\it Case $xyz=0$.} We use  the notation from  the  previous case. Clearly  the  set of points on  $\mathcal{C}$ with coordinates defined over $S$  and satisfying $xyz=0$ is $\mathcal{Z}_1 \cup \mathcal{Z}_2$. Based on our previous  discussion, we have that  $\#(\mathcal{Z}_1 \cup \mathcal{Z}_2)=4+\eta$. Hence there will be    $n(4+\eta)$ $\F_{q}$-rational points on $\mathcal{X}\cap \{xyz=0\}$. 
\end{enumerate}
Finally, adding the number $n(4+\eta)$ to the one  given in (\ref{afins}) yields (\ref{pontos}), and    finishes  the proof.
\end{proof}

\begin{exa}\label{exe2}
Consider the curve $\xx:x^{20}+2x^{10}y^{10}-y^{20}+z^{20}=0$ over $\F_{19^2}$. Note that $\xx$ has degree $d=2n$, where $n=\frac{19^2-1}{2(19-1)}$. Since the equation $x^{2}+2x-1=0$ has no $\F_{19}$-rational roots, Theorem \ref{nr pontos2} gives  $N_q(\xx)=3640$. 
\end{exa}

\begin{rem}
Note that, in contrast to the $\fq$-Frobenius classical case, the number $N_q(\xx)$ in examples \ref{exe1} and \ref{exe2}  exceed the upper bound in (\ref{sv con}).
\end{rem}

\appendix
\section{ A special family of plane quartics}

In what follows, we  note some simple  facts regarding the irreducibility of  certain  plane quartics that  are used in some of the proofs of  this paper. Despite the
simplicity, their detailed proofs can be quite lengthy. Thus  for the sake of brevity,  in some cases we omit the details and just indicate  the main steps.

Hereafter, we assume that $K$ is an algebraically closed field with $char(K)\neq 2$.

\begin{lem}\label{quart irr}
Let  $a,b,c,d,e,f \in K$ be such that  $\mathcal{Q}: a(xy)^2+b(xz)^2+c(yz)^2+xyz(dx+ey+fz)=0$
is a projective plane quartic. Then $\mathcal{Q}$ is irreducible if and only if
\begin{equation}\label{det}
abc \cdot \left| \begin{array}{ccc}
a & d/2 & e/2 \\
d/2 & b & f/2 \\
e/2 & f/2 & c \end{array} \right|\neq 0.
\end{equation}
\end{lem}
\begin{proof}
Consider the conic $\mathcal{C}: ax^2+by^2+cz^2+dxy+exz+fyz=0$ and assume condition (\ref{det}). This  implies that $\mathcal{C}$ is irreducible and does not  pass through any of the  points $(1:0:0)$, $(0:1:0)$, and $(0:0:1)$. Therefore,  the quartic  $\mathcal{Q}$ is the image of $\mathcal{C}$ by the standard Cremona transformation $(x:y:z:)\mapsto (xy:xz:zy)$. Hence $\mathcal{Q}$ is irreducible. The converse is trivial.
\end{proof}

For  $b,d,e \in K$, not all being zero, consider the plane projective quartic $\mathcal{Q}: F(x,y,z)=0$, where
\begin{equation}\label{eq 1}
F(x,y,z):=\Big((x +  y+  z)^2-b^2xy  -d^2 xz  + e^2yz \Big)^2-4\Big((bd - e)x - ey - ez\Big)^2yz.
\end{equation}

The idea is to find conditions on $a,b$, and $c$ for which the quartic  $\mathcal{Q}$ is irreducible. We begin with 
the following result, which states some basic  facts about the quartic $\mathcal{Q}$. The proof is trivial and will be omitted.

\begin{lem}\label{quart 2}
\begin{itemize}
\item[(i)] The polynomial $F$ defining the quartic $\mathcal{Q}$ satisfies 
\begin{equation}\label{eq2}\nonumber
F(x,y,z)=\Big((x +  y+  z)^2-e^2yz  -d^2 xz  + b^2xy \Big)^2-4\Big((ed - b)z - by - bx\Big)^2xy.
\end{equation}
 
 \item[(ii)] The points $P_1=(e^2 : d^2 : bde - d^2 - e^2)$, $P_2=(e^2 :bde -b^2 - e^2 : b^2)$ and $P_3=( bde - d^2-b^2:d^2:b^2)$ lie on $\mathcal{Q}$. Moreover, $P_1,P_2$, and $P_3$ are collinear if and only if  $$bde(b^2+d^2+e^2-bde)=0.$$
\end{itemize}
\end{lem}

\begin{thm}\label{main ap}
The quartic  $\mathcal{Q}$ is reducible if and only if at least two of the elements $b,d,e\in K$ are zero or $b^2 + d^2 + e^2 - bde=4$.
\end{thm} 
\begin{proof}If two of the elements $b,d,e\in K$ are zero, then the reducibility of  $\mathcal{Q}$ follows directly from (A.2) 
and Lemma A.2. If  $b^2 + d^2 + e^2 - bde=4$, then let $u,v \in K$ be such that $b=u+1/u$ and $e=v+1/v$. In this case, note that
$d=t+1/t$ where either $t=uv$ or $t=u/v$. From this, it can be checked that the factorization of  $F(x,y,z)$ is given by

                          $$F(x,y,z)=H(x,u^2y,t^2z)\cdot H(x,(1/u^2)y,(1/t^2)z),$$ 
where $H(x,y,z)=x^2+z^2+y^2-2(xy+xz+yz)$. To prove the converse, we consider the following three cases:

\begin{enumerate}
\item[(1)]\emph{$b^2 + d^2 + e^2 = bde$}. In this case, we have   $P_1=P_2=P_3=(e^2:d^2:b^2)$, and without loss of generality, assume $b\neq 0$. Dehomogenizing  $F(x,y,z)$ with respect to the variable $z$ and considering the following change of  variables $$f(x,y):=F(x-y+e^2/b^2,y+d^2/b^2,1),$$ we  focus on  the affine curve $ \mathcal{F}: f(x,y)=0$. Given the condition $b^2+d^2+e^2=bde$, it turns out that  $f(x,y)=f_4(x,y)+f_3(x,y)$,  where 
\begin{equation}\nonumber
f_4(x,y)=b^2x^4 - 2b^4x^3y + (b^6 + 2b^4)x^2y^2- 2b^6xy^3  + b^6y^4,
\end{equation} 
and  
\begin{eqnarray}\nonumber
f_3(x,y)&=& 4(bde-b^2d^2 )x^3 +4(2b^3de -b^4 - 2b^2e^2)x^2y \nonumber\\
        &+&4(b^4 - 2b^3de  + b^2d^2 + b^2e^2)xy^2.\nonumber
\end{eqnarray}
One  can check that resultant$(f_4(x,1),f_3(x,1))=b^{30}\neq 0$. Thus $\gcd(f_4,f_3)=1$, which implies
that $\mathcal{F}$ is an irreducible curve (see e.g. \cite[Problem 2.34]{Fu}).

\item[(2)]\emph{$b^2 + d^2 + e^2 - bde \neq 0,4$ and only one of the constants $b,d,e$ is zero}. Without loss of generality, we may assume $e = 0$, and therefore, $bd(b^2+d^2) \neq 0$. Setting  $u:=(b^2+d^2)/b^2$ and 
$$
M:=\left(\begin{array}{ccc}
-u & 0 & 0\\
u-1 & -1 & -u\\
1 & 1 & 0 \end{array} \right),
$$
we have that $\det M=-u^2\neq 0$. Let $T$ be the projective transformation associated to the matrix $M$ and define $G(x,y,z):=F(T(x,y,z))$. Dehomogenizing $G(x,y,z)$ with respect to the variable $z$, we find that the curve may be given by $f(x,y)=0$, where
\begin{eqnarray}
f(x,y)&=&\Big(y^2 + 2y +\frac{b^2(b^2 + d^2) + 4d^2}{b^2(b^2 + d^2)}\Big)x^2 \nonumber \\
      &-&\frac{2}{b^2}\Big(  (\frac{b^2 - d^2}{b^2 + d^2})y + 1\Big)x + \frac{1}{b^4}.\nonumber
\end{eqnarray}

Note that $f$ is a quadratic polynomial in $K(y)[x]$, which is reducible if  and only if its discriminant 
$$\Delta_{f}:= -\frac{16d^2}{b^2( b^2+d^2)^2}\Big(y^2 +(\frac{b^2+d^2}{b^2} )y +\frac{b^2+d^2}{b^4}\Big)$$
is a square in $K(y)$. This condition is equivalent to the discriminant of  $g(y)=y^2 +(\frac{b^2+d^2}{b^2} )y +\frac{b^2+d^2}{b^4}$, namely $\Delta_{g}:=\frac{(d^2+b^2)(d^2+b^2-4)}{b^4}$, being zero. Hence the result follows.

\item[(3)] \emph{$b^2 + d^2 + e^2 - bde \neq 0,4$ and $bde \neq 0$} . By Lemma \ref{quart 2}, the  points $P_1=(e^2 : d^2 : bde - d^2 - e^2)$, $P_2=(e^2 :bde -b^2 - e^2 : b^2)$, and $P_3=( bde - d^2-b^2:d^2:b^2)$ lie on $\mathcal{Q}$ and are not collinear. Consider the  projective change of coordinates mapping $P_1,P_2$, and $P_3$ to $(1:0:0), (0:1:0)$, and $(0:0:1)$, respectively. Based on this map, it can be checked that $\mathcal{Q}$ is projectively equivalent to the quartic defined by the equation
$$D(e^4x^2y^2 + d^4x^2z^2 + b^4y^2z^2)+2xyz(Ax+By+Cz)=0,$$
where 
$A=e^2d^2(bde+b^2  - d^2 - e^2),$
$B=e^2b^2(bde -b^2 + d^2 - e^2),$
$C=d^2b^2(bde -b^2  - d^2 + e^2),$ and
$D:=b^2 + d^2 + e^2 - bde$.
Since
$$
  \left| \begin{array}{ccc}
De^4 & A & B \\
A & Dd^4 & C \\
B & C & Db^4 \end{array} \right|=(bde)^6(b^2 + d^2 + e^2 - bde-4)\neq 0,
$$
Lemma \ref{quart irr} implies that $\qq$ is irreducible.
\end{enumerate}

\end{proof}

\subsection*{Acknowledgments}
The first and the second author were partially supported by FAPESP-Brazil, grants 2013/00564-1 and 2011/19446-3, respectively.


\begin{thebibliography}{HD}






\normalsize
\baselineskip=17pt


\bibitem{Bo3}H. Borges,
\emph{Frobenius nonclassical components of  curves with separated variables}, preprint.


\bibitem{Bo2}H. Borges,
\emph{On multi-Frobenius nonclassical plane curves}, Arch. Math. (Basel) {93} (2009), no. 6, 541--553.


\bibitem{FM}M.D. Fried and R.E. MacRae
\emph{On curves with separated variables},
Math. Ann. {180} (1969) 220--226.


\bibitem{Fu}W. Fulton,
 \emph{Algebraic curves},
 W. A. Benjamin INC, New York, 1969.


\bibitem{Ga}A. Garcia, 
\emph{The curves $y^n=f(x)$ over finite fields},
 Arch. Math. {54} (1990) 36--44.

\bibitem{GV2}A. Garcia and J.F. Voloch, 
\emph{Fermat Curves over finite fields},
Journal of Number Theory {30} (1988) 345--356.


\bibitem{GV1}A. Garcia and J.F. Voloch, 
\emph{Wronskians and linear independence in fields of prime characteristic},
 Manuscripta Math. {59} (1987) 457--469.

\bibitem{GL} M. Giulietti,
\emph{On the number of chords of an affinely regular polygon passing through a given point},
 Acta Scientiarum Mathematicarum (Szeged) {74}(3-4) (2008) 901-–913.

\bibitem{HV}A. Hefez and J.F. Voloch,
\emph{Frobenius non-classical curves},
 Arch. Math. {54} (1990) 263--273.

\bibitem{Hi}J.W.P. Hirshfeld,
\emph{Projective geometries over finite fields},
2nd edn, Oxford University Press, 1998.

\bibitem{HKT}J.W.P. Hirshfeld, G. Korchm\'aros and F.Torres,
\emph{Algebraic curves over a finite field},
Princeton Series in Applied Mathematics, 2008.

\bibitem{Ha}M. Homma,
\emph{Funny plane curves in characteristic $p>0$},
 Comm. Algebra {15} (1987) 1469--1501.

\bibitem{LN}R. Lidl and H. Niederreiter,
\emph{Finite fields},
Cambridge University Press, 1988.


\bibitem{Na}M. Namba,
\emph{Geometry of projective algebraic curves},
Pure and Applied Mathematics, Marcel Dekker Inc. 1984.

\bibitem{Pa}R. Pardini,
\emph{Some remarks on plane curves over fields of positive characteristic},
Composito Math. {60} (1986) 3--17.

\bibitem{St}H. Stichtenoth,
\emph{Algebraic function fields and codes},
 Springer, Berlin, 1993.

\bibitem{SV}K.O. St\"ohr and J.F. Voloch, 
\emph{Weierstrass points on curves over finite fields}, 
Proc. London Math. Soc. {52} (1986) 1--19.




\end{thebibliography}
\end{document}